\documentclass[12pt,letterpaper]{amsart}
\usepackage{geometry}
\usepackage{mathpazo} 
\usepackage{amsmath,amssymb,amsfonts,amsthm,amscd}
\usepackage{mathrsfs,latexsym,url,setspace,xcolor,graphicx}
\newtheorem{theorem}{Theorem}
\newtheorem{theoremA}{Theorem}
\newtheorem*{theorem*}{Theorem}

\newtheorem{proposition}{Proposition}
\newtheorem{lemma}{Lemma}

\theoremstyle{remark}

\newtheorem{remark}{Remark}

\newcommand{\spec}{\text{Spec}}
\newcommand{\ad}{\text{ad}}

\newcommand{\gr}{\text{gr}}

\usepackage[author={S},color=yellow,opacity=0.6,markup=Highlight]{
pdfcomment}
\onehalfspace

\title{The Weyl algebra as a fixed ring}

\author{Akaki Tikaradze}
\email{ tikar06@gmail.com}
\address{University of Toledo, Department of Mathematics \& Statistics, 
Toledo, OH 43606, USA}

\date{\today}
\begin{document}

\begin{abstract}
We prove that the Weyl algebra over $\mathbb{C}$ cannot be a fixed ring of any domain under a nontrivial action of a finite
group by algebra automorphisms, thus settling a 30-year old  problem. In fact, we prove the following much more general result.
Let $X$ be a smooth affine variety over  $\mathbb{C}$,  let $D(X)$ denote the ring of algebraic differential
operators on $X,$ and let $\Gamma$ be a finite group.  If $D(X)$ is isomorphic
to the ring of $\Gamma$-invariants  of a $\mathbb{C}$-domain $R$ on which  $\Gamma$ acts faithfully by
$\mathbb{C}$-algebra automorphisms, then $R$ is isomorphic to the ring of differential operators on a $\Gamma$-Galois covering
of $X.$ 

\end{abstract}
\maketitle
\begin{center}
\emph{Dedicated to my parents Natalia and Soso Tikaradze with love and admiration}
\end{center}

\section{Introduction}

Throughout given a variety $X,$ by $\mathcal{O}(X)$ we will denote the algebra of its global functions.
Also, given a ring $R$, its center will be denoted by $Z(R).$
Let $S$ be a commutative Noetherian ring.
Given a smooth affine variety $X$ over $S,$ we will denote by $D(X)$ its ring of differential operators over $S,$ as usual.
Recall that  $D(X)$ is generated over $\mathcal{O}(X)$ by the Lie algebra of its $S$-derivations $T^1_X=Der_S(\mathcal{O}(X), \mathcal{O}(X)),$
with the following relations 
$$ [\tau, f]=\tau (f), \tau\tau_1-\tau_1\tau=[\tau_, \tau_1], \quad f\in \mathcal{O}(X), \tau, \tau_1\in T_X.$$
Also, by $W_n(S)$ we will denote the $n$-th Weyl algebra over $S.$
Whether the Weyl algebra over $\mathbb{C}$ can be a  fixed point ring of a $\mathbb{C}$-domain under a nontrivial
action of a finite group has been an open problem  in ring theory for some time now.
In this regard it was proved by Smith \cite{S} that if $R$ is a $\mathbb{C}$-domain and $\Gamma\subset Aut_{\mathbb{C}}(R)$
is a finite solvable group of $\mathbb{C}$-algebra automorphisms, such that
 $R^{\Gamma}=W_1(\mathbb{C}),$ then $R=R^{\Gamma}.$
This result was generalized by Canning and Holland \cite{CH} to the ring of differential operators
of an affine smooth algebraic curve over $\mathbb{C}.$ Namely, they showed that if $R$ is a $\mathbb{C}$-domain and $\Gamma$ is a finite solvable group
of $\mathbb{C}$-automorphisms of $R$ such that $R^{\Gamma}\cong D(X),$ where $X$ is a smooth affine curve over $\mathbb{C},$
then there exists a smooth affine curve  $Y$ over $\mathbb{C}$, such that $Y$ is a $\Gamma$-Galois covering of $X$  and $R\cong D(Y).$ 
The proofs are based on the description of the Picard group of invertible bimodules on $D(X).$
No such result is available for high dimensional $X.$ Moreover these proofs do not work
for non-solvable $\Gamma.$

On the other hand Alev and Polo \cite{AP} showed that if $\Gamma$ is an arbitrary finite group
such that if both $R, R^{\Gamma}$ are isomorphic to the $n$-th Weyl algebra $W_n(\mathbb{C})$, then again $R=R^{\Gamma}.$
Similarly they proved that the enveloping algebra of a semi-simple Lie algebra over $\mathbb{C}$ cannot occur
as a fixed point ring of a nontrivial finite group action on an enveloping algebra of a semi-simple Lie algebra.

Recall that given an affine variety $X$ over an algebraically closed field $\bold{k}$, its $\Gamma$-Galois covering
is an affine (possibly disconnected) variety $Y$ together with a $\Gamma$-equivariant  \'etale covering $f:Y\to X,$
such that $\Gamma$ acts simply transitively on the fibers of closed points of $X.$

The main results of this paper  represents the strongest possible
generalization of the above results.

\begin{theorem}\label{Main}

Let $X$ be a smooth affine variety over $\mathbb{C}.$ 
Let $R$ be a $\mathbb{C}$-domain, let $\Gamma\subset Aut_{\mathbb{C}}(R)$ be a finite group of
automorphisms of $R.$  If $R^{\Gamma}=D(X)$, then there exists a $\Gamma$-Galois covering $Y\to X$ such that $R\cong D(Y)$ as $\Gamma$-algebras.

\end{theorem}

As an immediate corollary we get that if $X$ is a smooth affine variety over $\mathbb{C}$ such that
$X(\mathbb{C})$ is simply connected, then $D(X)$, in particular the Weyl algebras $W_n(\mathbb{C})$, cannot be a fixed point ring of
any $\mathbb{C}$-domain under a nontrivial finite group action.

To state our next result we will recall some terminology.
Let $S$ be a commutative ring.
Recall  that an $S$-algebra $R$ is a Galois covering of an algebra $B$ with Galois group $\Gamma$
if $\Gamma$ acts on $R$ via $S$-algebra  automorphisms, 
$B=R^{\Gamma}$ and the skew group ring  $R\rtimes \Gamma$ is isomorphic to $End_{B}(R);$ 
moreover $R$ is a finitely generated projective generator as a $B$-module.
We will say that a $\Gamma$-Galois covering $R$ of an algebra $B$ is a trivial covering if $R\cong \Pi_{g\in \Gamma} B[e_g]$
where $e_g$ are primitive orthogonal central idempotents, $\sum_{g\in \Gamma} e_g=1$ and $\Gamma$ acts on $e_g$ by (left) multiplication on indices.

Let $\mathfrak{g}$ be a semi-simple Lie algebra over $\mathbb{C}.$ Let $U\mathfrak{g}$ be its enveloping algebra.
We will say that a central character $\chi:Z(U\mathfrak{g})\to \mathbb{C}$ is very generic
if its values on the standard generators of $Z(U\mathfrak{g})$ are algebraically independent over $\mathbb{Z}.$
Denote by $U_{\chi}\mathfrak{g}$ the central quotient of $U\mathfrak{g}$ corresponding to $\chi.$
\begin{theorem}\label{g}
 Let $\mathfrak{g}$ be a semi-simple Lie algebra over $\mathbb{C}.$ Then for a very generic central character $\chi$ of $U\mathfrak{g}$,
 there is no integral domain $R$ with a nontrivial finite group of automorphisms $\Gamma\subset Aut_{\mathbb{C}}R$
 so that $R^{\Gamma}=U_{\chi}\mathfrak{g}.$

\end{theorem}

Notice that it is necessary to require for the central character $\chi:Z(U\mathfrak{g})\to \mathbb{C}$ to be
rather generic for the above result to hold, see \cite{S1}.

If $\Gamma$ is a finite subgroup of algebra automorphisms of $R$, and
 $H$ is a normal subgroup of $\Gamma,$ then as $R^{\Gamma}=(R^{H})^{\Gamma/H}$, proofs of above theorems
easily reduce to the case of a simple group $\Gamma.$
Thus we will be assuming from now on that $\Gamma$ is a simple group.

The proof is  based on the reduction $\mod p$ technique for a large enough prime $p.$


\section{Preliminary results}

At first we will need to recall a fundamental observation due to Bezrukavnikov, Mirkovic and
Rumynin [\cite{BMR}, Theorem 2.2.3] which asserts that given a smooth variety $X$ (which for us will always be affine) over
an algebraically closed field $\bold{k}$ of characteristic $p>0$, then its ring of (crystalline) differential
operators $D(X)$ is an Azumaya algebra over the Frobenius twist of the cotangent bundle $T^{*}(X).$
Namely, given $ \theta\in T_X=Der_{\bold{k}}(\mathcal{O}(X), \mathcal{O}(X)),$ denote by ${\theta}^{[p]}\in T_X$ the $p$-th power of derivation
$\theta.$ Then the center of $D(X)$ is generated by $$f^p, \theta^p-\theta^{[p]},\quad f\in \mathcal{O}(X),\quad \theta\in T_X.$$

We will start by establishing an easy commutative version of Theorem \ref{Main}.
\begin{lemma}\label{Poisson}
Let $X$ be a smooth affine variety over  an algebraically closed field $\bold{k}.$  
Let  $f:Y\to T^{*}(X)$ be a Galois covering with Galois group $\Gamma.$ Assume that $p=char(\bold{k})$ does not divide $|\Gamma|.$
Then there exists an affine variety $Y'$ equipped with a Galois covering $f':Y'\to X$ with Galois group 
$\Gamma,$ such that $Y\cong T^{*}(Y')$ interchanging $f,T^{*}(f')$.
\end{lemma}
\begin{proof}
Recall that the cotangent bundle $T^*(X)$ is a conic variety, which amounts to the fact that 
there is the natural contracting action of the multiplicative group $\mathbb{G}_m$ corresponding to the grading on
$\mathcal{O}(T^*(X)).$
We will construct a contracting action of $\mathbb{G}_m$ on $Y$ that is compatible with the action on $T^{*}(X).$
Let $x\in Y$ be a closed point, let $Z=\overline{\mathbb{G}_mf(x)}\cong \mathbb{A}_{\bold{k}}^1.$ Thus $f^{-1}(Z)\to Z$ is a Galois covering
with Galois group $\Gamma.$ Since $p$ does not divide $|\Gamma|,$ then
$f^{-1}(Z)\to Z$ must be a $|\Gamma|$-fold trivial covering, hence
 $f^{-1}(Z)=\bigsqcup \mathbb{A}_{\bold{k}}^1.$
Let $Z_x$ be the connected component of $f^{-1}(Z)$ containing $x$. Hence $Z_x\cong \mathbb{A}^1_{\bold{k}}.$
Thus we may equip $Z_x$ with the unique $G_m$-action such that $f:Z_x\to Z$ commutes with it.
Varying $x\in Y,$  we get an action of $\mathbb{G}_m$ on $Y$ commuting with $f.$ It follows that this action is a contraction
of $X$ on $f^{-1}(X)$ and has no negative eigenvalues. Hence $\mathcal{O}(Y)$ is a nonnegatively graded algebra
$$\mathcal{O}(Y)=\bigoplus_{i\geq 0}B_, \quad \spec (B_0)=f^{-1}(X).$$
Put $Y'=f^{-1}(X).$
 Next we claim that
$T^{*}(Y')\cong Y.$ Indeed, put $B'=B_0[T_X].$ Then $\spec (B')\to T^{*}(X)$ is a $\Gamma$-Galois
covering, hence $Y=\spec (B')$ and $\spec (B')\cong T^{*}(Y').$
\end{proof}

Next we will recall the Howlett-Isacs theorem \cite{HI} which will be used crucially in Lemma \ref{twisted}.
Let $G$ be a finite group, $F$ an algebraically closed field whose characteristic does not
divide $|G|.$ Let $\rho\in Z^2(G, F^*)$ be a two-cocycle, and let $F_{\rho}[G]$ denote the corresponding
$\rho$-twisted group algebra:
$$F_{\rho}[G]=\oplus_{g\in G}Fe_g,\quad e_{g_1}\cdot e_{g_2}=\rho(g_1,g_2)e_{g_1g_2}, \quad g_1, g_2\in G.$$
 Then the  Howlett-Isaacs theorem asserts that if $G$ is a
simple group, then for any 2-cocycle $\rho$, $F_{\rho}[G]$ is not a simple $F$-algebra.
\footnote{We thank Professor Pham Tiep for telling us about the Howlett-Isaacs theorem.}

Now we have the following key

\begin{lemma}\label{twisted}
Let $Z$ be a commutative domain over an algebraically closed field $\bold{k}$ of characteristic $p$, and let $A$ be a finite algebra
over $Z,$  such that $Z(A)=Z$ and $A$ has no $Z$-torsion.  Let $\Gamma\to Aut_Z(A)$
be a finite simple group of central automorphisms of $A,$ such that $p$ does not divide $|\Gamma|$. 
 Then $A\rtimes \Gamma$ is not an Azumaya algebra over $Z.$

\end{lemma}
\begin{proof}
Proof by contradiction. Throughout given a finite dimensional algebra $S,$ its Jacobson radical
will be denoted by $J(S).$ 
We will make a base change from $Z$ to an algebraically closed field $F$ containing $Z$ and
put $A_F=A\otimes_ZF.$ We conclude that
$Z(A_F)=F$ and $J(A_F)=0$ since $J( A_F\rtimes\Gamma)=0$. So  $A_F$ is
a central simple algebra over $F$, therefore $A_F$ is the matrix algebra $M_n(F)$ for some $n$.
Hence, we have that
 $$M_n(F)\rtimes\Gamma\cong M_{m}(F).$$
Let $\rho\in Z^2(\Gamma, F^{*})$ be a 2-ccocycle corresponding to the projective
representation $\Gamma\to Aut(M_{n}(F)).$ Thus $$F_{\rho}[\Gamma]\otimes_F M_{n}(F)\cong M_{m}(F),$$
\noindent where $F_{\rho}[\Gamma]$ is the twisted group algebra of $\Gamma$ corresponding to cocycle $\rho.$
Therefore $F_{\rho}[\Gamma]$ must be a central simple algebra over $F$,
which is a contradiction by the Howlett-Isaacs theorem.

\end{proof}









We will make use of the following lemma  which is an immediate corollary of
a result by Brown and Gordon [\cite{BG} Theorem 4.2].
At first let us recall their definition of Poisson orders.

Let $A$ be a Poisson algebra over $\mathbb{C}.$ Then a Poisson
order over $A$ is an algebra $B$ containing $A$ as a central subalgebra such
that it is finitely generated as a $A$-module, and  is equipped with a $\mathbb{C}$-linear map $A\to Der_{\mathbb{C}}(B, B)$
satisfying the Leibnitz identity, such that it restricts to the Poisson bracket on $A$.

\begin{lemma}\label{BGord}
Let $ A$ be an affine Poisson $\mathbb{C}$-algebra such that
 the Poisson bracket on $A$ induces a structure of a smooth symplectic variety on $\spec (A).$
 Let $B$ be a Poisson order over $A.$ Then for all $t\in\spec (A),$ algebras
$B/tB$ are mutually isomorphic. In particular, the restriction map $\spec (B)\to \spec (A)$ is an \'etale covering.
\end{lemma}


 We will need the following result which follows immediately from \cite{M}.

\begin{lemma}\label{Susan}
Let $A$ be a simple Noetherian domain over $\mathbb{C}$ such that $Z(A)=\mathbb{C}.$
Let $R$ be a $\mathbb{C}$-domain  equipped with a faithful $\Gamma$-action, such that $R^{\Gamma}=A.$
Then $R$ is a $\Gamma$-Galois covering of $A.$

\end{lemma}

\begin{proof}
It follows that $R$ is a simple domain: If $I$ is a nonzero proper
two-sided ideal of $R$, then $ J=\bigcap_{g\in \Gamma} g(I)$ is a $\Gamma$-invariant nonzero two-sided ideal in $R.$
 Then $J^{\Gamma}\neq 0$ since
$J$ is not nilpotent as follows from [\cite{M}, Theorem 1.7]. Hence $J\cap R \neq 0$ is a proper nonzero two-sided ideal in $R$, a contradiction. 
It follows that $Z(R)=\mathbb{C}$ and $\Gamma$ acts on $R$ by outer automorphisms: If $g\in \Gamma$ such that
$g(x)=axa^{-1}$ for some $a\in R,$ then $a^{|\Gamma|}\in Z(R)$, but since $Z(R)=\mathbb{C},$ we get that
$a^{|\Gamma|}\in \mathbb{C}$. So $a$ is a root of unity and $g$=Id.
Hence  by [\cite{M}, Theorem 2.5],  $R$ is a $\Gamma$-Galois covering of $A.$
\end{proof}

Finally we have the following crucial

\begin{proposition}\label{Azu}
Let $Z$ be an affine commutative domain over an algebraically closed field $\bold{k}$ of characteristic $p.$
let $A$ be an Azumaya algebra over $Z$, and let $R$ be a $\Gamma$-Galois covering of $A.$
Assume that $p$ does not divide $|\Gamma|.$
Then  $\spec (Z(R))\to \spec (Z)$ is a $\Gamma$-Galois covering and $R=A\otimes_ZZ(R).$

\end{proposition}
\begin{proof}
Let us put $Z(R)=Z_1.$
Remark that $Z=Z_1^{\Gamma}.$
Indeed, since $ R\rtimes\Gamma$ is an Azumaya algebra over $Z$, then $Z(R\rtimes\Gamma)=Z,$
in particular $R$ commutes with $Z,$ therefore $Z\subset Z_1$. Hence $Z_1^{\Gamma}=Z.$

Since $R$ is a $\Gamma$-Galois covering of $A$, it follows that $R$ is a projective left, right $A$-modules of rank $|\Gamma|.$
Since $Z\subset Z_1,$ it follows that $R$ is
a module over $ A\otimes_{Z}A^{op}$. As $A$ is an Azumaya algebra over $Z$,
it follows that $R\cong A\otimes_{Z}B,$ where $B$ is the centralizer of
$A$ in $R.$ Thus we conclude  that $B\rtimes\Gamma$ is an Azumaya
algebra over $Z$ of rank $|\Gamma|^2.$ Also $Z(B)=Z_1.$
It follows from Lemma\ref{twisted} that the action of $\Gamma$ on $Z_1$ must be faithful.
Also it is clear that  $Z_1$ has no nilpotent elements.
Let $ \eta\in \spec (Z)$ be the generic point. Then $Z(B_{\eta})=(Z_1)_{\eta}$ and $B_{\eta}$
 is a $|\Gamma|$-dimensional semi-simple $Z_{\eta}$-algebra. On the other hand,
 since $\Gamma$ acts faithfully on $(Z_1)_{\eta}$ and $(Z_1)^{\Gamma}_{\eta}=Z_{\eta},$ it follows that 
$(Z_1)^{\Gamma}_{\eta}$ is $|\Gamma|$-dimensional over $(Z)_{\eta}.$
Therefore $(Z_1)_{\eta}=B_{\eta}.$  Hence $B=Z_1,$ since $B$ has no $Z$-torsion.

Now we claim that $\spec (Z_1)\to \spec (Z)$ is an \'etale covering with Galois group $\Gamma.$
Indeed, for any $\chi\in \spec (Z)$, we have that $(Z_1)_{\chi}\rtimes\Gamma$ is a matrix
algebra. Therefore $J((Z_1)_{\chi})=0$. Since $\dim (Z_1)_{\chi}=|\Gamma|,$ it follows that 
$$(Z_1)_{\chi}=\bold{k}\times\cdots\times\bold{k}.$$
Hence $\spec (Z_1)\to \spec (Z)$ is an \'etale covering with Galois group $\Gamma,$ as desired.

\end{proof}

We will also need the following result on lifting of $p'$-order Galois coverings from characteristic $p$ to 
characteristic 0 \footnote{I am grateful to D.Harbater, especially
P.Achinger and A.Javanpeykar for the proof.}.
It will only be used in the proof of Theorem \ref{g}, not for Theorem \ref{Main}.
\begin{lemma}\label{lifting}
Let $X$ be a smooth geometrically connected affine variety over $S$, where $S\subset \mathbb{C}$ is a finitely generated ring. 
Suppose that a finite group $\Gamma$ appears as a quotient of the \'etale fundamental group of $X_{\bold{k}}$  for all
large enough primes $p>>0$, where $X_{\bold{k}}=X\times_{\spec (S)}\spec (\bold{k})$ is a $\mod p$ reduction of $X$
by a base change $S\to \bold{k}$ to an algebraically closed field of characteristic $p.$
Then $\Gamma$ is a quotient of the  fundamental group of $X(\mathbb{C}).$

\end{lemma}
\begin{proof}

This follows from general results about \'etale coverings tamely ramified across normal crossing divisors.
 Indeed, let $\overline{X}$ be a good compactification of $X_{\mathbb{C}}$, thus it is a smooth
projective variety such that $\overline{X}\setminus X_{\mathbb{C}}$ is a normal crossings divisor.
By enlarging $S$ we may assume that $\overline{X}$ is defines over $S$ and
$\overline{X}\setminus X$ is a normal crossings divisor over $S$. For $p>>0$, for a base change $S\to \bold{k},$
$\overline{X_{\bold{k}}}\setminus X_{\bold{k}}$ is a divisor with normal crossings.
Let us choose an embedding $S\to V$, where $V$ is a complete discrete valuation ring with the
residue field $\bold{k}$-an algebraically closed field of characteristic $p.$
For example we may take $V$ to be the ring of Witt vectors over $\bold{k},$ where $\bold{k}$
is large enough algebraically closed field of characteristic $p.$
Then it follows from [\cite{LO}, Corollary A.12] that $X_V$ has a connected Galois covering
with the Galois group $\Gamma,$ 
Applying a base change $V\to \mathbb{C}$ we obtain
the desired $\Gamma$-Galois covering of $X_{\mathbb{C}}.$
\end{proof}



\section{The proof of the main result for rings of differential operators}

We start by the following Lemma. It shows that to prove Theorem \ref{Main}, it will
suffice to prove that $R$ as a bimodule over $D(X)$ is supported on the diagonal of $X\times X.$
\begin{lemma}\label{K}
Let $X$ be a smooth affine variety over $\mathbb{C},$ and let $R$ be a $\mathbb{C}$-domain.
Assume that a finite simple group $\Gamma$ acts faithfully on $R$ so that $R^{\Gamma}=D(X).$ 
Moreover assume that for any $f\in \mathcal{O}(X), \ad(f)=[f, -]$ acts locally nilpotently on $R.$
Then $R=D(Y),$ where $Y$ is a $\Gamma$-Galois covering of $X.$
\end{lemma}
\begin{proof}
We may view $R$ as a left module over $D(X\times X)$. Assumptions above imply
that $R$ is supported on the diagonal $\Delta(X)\subset X\times X.$
Let $R'$ be the centralizer of $\mathcal{O}(X)$ in $R.$ Clearly $R'$ is preserved by $\ad(y)$ for all $y\in Der_{\mathbb{C}}(\mathcal{O}(X),\mathcal{O}(X))=T_X.$
Thus $R'\otimes_{\mathcal{O}(X)}D(X)$ is naturally a $D(X)$-bimodule equipped with the bimodule map $R'\otimes_{\mathcal{O}(X)}D(X)\to R.$
Now recall that Kashiwara's theorem [\cite{GM}, Theorem 4.9.1] establishes an equivalence between the category
of $D(X\times X)$-modules supported on the diagonal and the category of $D(X)$-modules, given by direct and inverse image functors
corresponding to the diagonal embedding $\Delta:X\to X\times X.$ Applying this to $R$,
 we get that $R\cong R'\otimes_{\mathcal{O}(X)}D(X).$
Next we will introduce an ascending filtration on $R$ as follows. Let $$R_n=\lbrace x\in R,\quad \ad(f)^{n+1}(x)=0,\forall f\in \mathcal{O}(X)\rbrace, n\geq 0.$$
Then $$R_0=R', \quad R_1=R'T_X=T_XR', \quad R_nR_m\subset R_{n+m}.$$ 
This filtration restricts on $D(X)$ to the usual filtration by the order of differential operators. It follows that
$\gr R=R'\otimes_{\mathcal{O}(X)}\mathcal{O}(T^*(X))$ is a finitely generated module over $\mathcal{O}(T^{*}(X))=\gr (D(X)),$ moreover
$\gr R$ is a Poisson order over $\mathcal{O}(T^{*}(X)).$ Now lemma  \ref{BGord} implies that for all $\eta\in T^{*}(X)$, algebras
$(\gr R)_{\eta}$ are isomorphic to each other. Therefore, for all $\eta\in X,$ algebras $R'_{\eta}$
are mutually isomorphic. Put $Y=\spec (Z(R')).$ We claim that $Y\to X$ is finite \'etale map and $R'$
is an Azumaya algebra over $Y.$ Indeed, since for any $y\in T_X$ $$\ad(y)(\mathcal{O}(Y))\subset \mathcal{O}(Y),$$
then $R''=Z(R')\otimes_{\mathcal{O}(X)}D(X)$ is a $\Gamma$-invariant subalgebra of $R$,
and $R''^{\Gamma}=D(X).$
Thus applying the above argument to $R''$ instead of $R$ we get that  $\spec (Z(R'))\to X$ is a finite \'etale map.
Then we have a homomorphism of algebras $\theta:R''=\mathcal{O}(Y)\otimes_{\mathcal{O}(X)}D(X)\to D(Y)$ 
 given by mapping $y\in T_X$ to $ \ad(y)\in Der_{\mathbb{C}}(\mathcal{O}(Y), \mathcal{O}(Y))=T_{Y}.$
As $T_Y=T_X\otimes_{\mathcal{O}(X)}\mathcal{O}(Y),$ it follows that $\theta$ is an isomorphism: $R''=D(Y)$ and as a bimodule over $D(Y), R$ is supported on the diagonal of
$Y\times Y$. Hence in the above argument we may replace $D(X)$ by $D(Y)$ to conclude that $R'$ is an Azumaya agebra over $Y$.
Notice that up to now we have not used the action of $\Gamma$ on $R.$
 
Since $D(X)$ is a simple ring, it follows from lemma \ref{Susan} that $R$ is a $\Gamma$-Galois covering
of $D(X).$ Hence $R$ is a projective $D(X)$-module of rank $|\Gamma|$. Therefore $R'$ is also a projective $\mathcal{O}(X)$-module
of rank $|\Gamma|.$
  If the action of $\Gamma$ on $Z(R')$ is faithful then $Y\to X$ is a $\Gamma$-Galois covering, hence $R''=R$ and we
are done. Thus we may assume that $Z(R')=\mathcal{O}(X).$  Then $R'^{\Gamma}=\mathcal{O}(X)$ and
$\Gamma\subset Aut_{\mathcal{O}(X)}R'.$ Since $R\rtimes\Gamma$ is Morita equivalent to
$D(X),$ It follows that $R\rtimes\Gamma'$ is Morita equivalent to
$\mathcal{O}(X)$, which is a contradiction by  Lemma \ref{twisted}.
\end{proof}
\begin{remark}
It follows from the above proof that if $X$ is a smooth algebraic curve and $R$ is a $\mathbb{C}$-domain
containing $D(X)$, such that it is finite as both left and right $D(X)$-module, and $R$ is supported on the
diagonal of $X\times X$ as a $D(X)$-bimodule, then $R\cong D(Y)$, where $Y\to X$ is a finite \'etale map.
Indeed, this is because the Azumaya algebra in the proof must be split as the Brauer group over curves
is trivial over algebraically closed fields.
\end{remark}

\begin{proof}[Proof of Theorem \ref{Main}.]
As in the proof above, since $D(X)$ is a simple ring, it follows from \ref{Susan} that $R$ is a $\Gamma$-Galois covering
of $D(X).$ There exists finitely generated $\mathbb{Z}$- algebra $S\subset \mathbb{C}, \frac{1}{|\Gamma|}\in S,$
such that $X=X'_\mathbb{C},$ where $X'$ is a smooth affine variety over $S,$
and an $S$-subalgebra $R'$ of $ R$ such that $R=R'\otimes_{S}\mathbb{C},$ $\Gamma$ acts on $R'$
by $S$-automorphisms and $R'$ is  a $\Gamma$-Galois covering of $D(X').$

Therefore  for all large primes $p>>0$ there exists
a homomorphism $\rho:S\to \bold{k}$, where $\bold{k}$ is an algebraically closed field of characteristic $p$, such
that $R_{\bold{k}}=R'\otimes_S\bold{k}$ is a $\Gamma$-Galois covering of $D(X_{\bold{k}}),$ where $X_{\bold{k}}=X'\times _{\spec (S)}\spec (\bold{k}).$
For simplicity we will denote by $Z_0, Z_1$ the centers of $D(X_{\bold{k}}), R_{\bold{k}}$ respectively. 
Recall that $D(X_{\bold{k}})$ is an Azumaya algebra over $Z_0$, and
$\spec (Z_0)=T^{*}(X_{\bold{k}}).$

It follows from Proposition \ref{Azu} that $\Gamma$ acts faithfully on $Z_1, Z_0=Z_1^{\Gamma}$ and $\spec (Z_1)\to \spec (Z_0)$
is a $\Gamma$-Galois covering.
Applying Lemma\ref{Poisson}  to  the $\Gamma$-Galois covering $\spec (Z_1)\to \spec (Z_0)$, we conclude
 that there exists an affine $\bold{k}$-variety $Y_{\bold{k}}$
and an \'etale covering $Y_{\bold{k}}\to X_{\bold{k}}$ with the Galois group $\Gamma,$ such that
 $\spec (Z_1)\cong T^{*}(Y_{\bold{k}})$ as $\Gamma$-varieties.
Since $$D(Y_{\bold{k}})\cong D(X_{\bold{k}})\otimes_{Z_0}Z(D(Y_{\bold{k}})),$$
and $$Z(D(Y_{\bold{k}}))\cong \mathcal{O}(T^{*}(Y_{\bold{k}})),$$
 it follows that
$D(Y_{\bold{k}})\cong R_{\bold{k}}$ as $\Gamma$-algebras.

Next we will show that if $f\in\mathcal{O}(X'),$ then $\ad(f)$ is locally nilpotent on $R'.$
In view of Lemma \ref{K}, this will yield the theorem.
Let $x\in R'.$
Let $f_1,\cdots, f_m\in D(X')$ be such that $$x^m+\sum_{i=0}^{m-1} f_ix^{m-i}=0 .$$
Let $N=Max_i(\deg(f_i))$, here $\deg$ means the degree of a differential operator in $D(X')$.
We claim that $\ad(f)^{N+1}(x)=0$. We will show this by proving that $\ad(\bar{f})^m(\bar{x})=0$ in $R_{\bold{k}}$
for all large enough $p,$ where $\bar{y}$ denotes the image of an element $y\in R'$ under the base change map
$R'\to R_{\bold{k}}.$
Since $R_{\bold{k}}\cong D(Y_{\bold{k}}),$  we will equip $R_{\bold{k}}$ with the filtration
corresponding to the degree filtration of differential operators in $D(Y_{\bold{k}})$.
Hence $\gr R_{\bold{k}}=\mathcal{O}(T^{*}(Y_{\bold{k}})).$
Clearly this filtration restricts on $D(X_{\bold{k}})$ to the filtration corresponding to the order of differential
operators on $X_{\bold{k}}.$ Thus $\gr (\bar{f_i})\leq N, 1\leq i\leq m.$ 
Now it  follows that $\gr R_{\bold{k}}$ has no nilpotent elements
and is a torsion free $\gr D(X_{\bold{k}})=\mathcal{O}(T^{*}(X_{\bold{k}}))$-module.
Thus we have that 
$$\gr(\bar{x})^m=-\sum_{i=1}^m \gr(\bar{f_i})\gr(\bar{x})^{m-i}.$$
This implies that $\gr(\bar{x})\leq N.$
Hence for any $g\in \mathcal{O}(X_{\bold{k}}),$ we have $\ad(g)^m(\bar{x})=0.$ In particular,
$\ad(\bar{f})^m(\bar{x})=0.$ Therefore, $\ad(f)^m(x)=0$ and we are done.

\end{proof}

\section{The proof for enveloping algebras}

Before proving Theorem \ref{g} we will need to recall some standard facts about enveloping
algebras of semi-simple Lie algebras in characteristic $p>0.$
Let $G$ be a semi-simple, simply connected algebraic group of rank $n$ over $\mathbb{C}$, let $\mathfrak{g}$
be its corresponding Lie algebra. Let $p>>0$ be a very large prime,
and let $G_{\bf{k}}, \mathfrak{g}_{\bold{k}}$ denote  reductions of $G, \mathfrak{g}$ to an algebraically
closed field $\bold{k}$ of characteristic $p.$
We will be identifying $\mathfrak{g}_{\bold{k}}$ with $\mathfrak{g}_{\bold{k}}^{*}$ via a nondegenerate $G_{\bold{k}}$-invariant bilinear form
on $\mathfrak{g}_{\bold{k}}$, as usual.
 Let $f_1,\cdots, f_n$ be homogeneous generators
of $\bold{k}[\mathfrak{g}^*]^G=\bold{k}[f_1,\cdots, f_n].$ Let $\tilde{f_1},\cdots, \tilde{f_n}$ be
the corresponding generators of $(U\mathfrak{g}_{\bold{k}})^{G_{\bold{k}}}=\bold{k}[\tilde{f_1},\cdots, \tilde{f_n}].$
Given  $\chi \in  \spec (\bold{k}[\mathfrak{g}^*_{\bold{k}}]^{G_{\bold{k}}}),$ we will denote by $\mathcal{N}_{\chi}$
the preimage of $\chi$ under the map $F=(f_1,\cdots, f_n):\mathfrak{g}_{\bold{k}}\to \bold{k}^n.$
Of course for the origin $\chi=0, \mathcal{N}_0$ is the nilpotent cone of $\mathfrak{g}_{\bold{k}}.$ While 
for a generic $\chi, \mathcal{N}_{\chi}$ is the conjugacy class of a regular semi-simple element of $\mathfrak{g}_{\bold{k}}.$

Given $\chi\in \spec (\bold{k}[\mathfrak{g_{\bold{k}}}^*]^{G_{\bold{k}}}) $, we will say that $\chi$ is very generic if its 
values on $f_1,\cdots,f_n$ are algebraically independent over $F_p.$
Similarly, given $\chi\in \spec (U\mathfrak{g})^{G_{\bold{k}}}$ , we will say that it is very generic if
its values on $\tilde{f_1},\cdots,\tilde{f_n}$ are algebraically independent over $F_p.$
Let $\chi:U\mathfrak{g}_{\bold{k}}^{G_{\bf{k}}}\to \bold{k}$ be a character, and let
$U_{\chi}\mathfrak{g}_{\bold{k}}$ be the corresponding central reduction of $U\mathfrak{g}_{\bold{k}}.$
Then it is known that $Z(U_{\chi}\mathfrak{g}_{\bold{k}})\cong \mathcal{O}(\mathcal{N}_{\tilde{\chi}})$
for the corresponding $\tilde{\chi}.$ Moreover for a very generic $\chi$
character $\tilde{\chi}$ is also very generic, and $ U_{\chi}(\mathfrak{g}_{\bold{k}})$
is an Azumaya algebra. Indeed, 
Let $V$ be a simple $U_{\chi}(\mathfrak{g}_{\bold{k}})$-module. Then the $p$-character of $V$
viewed as a simple $\mathfrak{U}\mathfrak{g}$-module is regular. Hence $\dim V=p^d$, where
$d$ is half the dimension of the nilpotent cone of $\mathfrak{g}_{\bold{k}}.$
Therefore all simple $U_{\chi}(\mathfrak{g}_{\bold{k}})$-modules have the same dimension.
So  $U_{\chi}(\mathfrak{g}_{\bold{k}})$ is an Azumaya algebra by [\cite{BG1}, Proposition 3.1].

As varieties $\mathcal{N}_{\chi}$ are simply connected for very generic characters over $\mathbb{C},$
Lemma \ref{lifting} yields the following
 
 \begin{lemma}\label{Congugacy}
Let $G, \mathfrak{g}$ be as above and $p>>0.$ 
 Let $\chi:\bold{k}[\mathfrak{g_{\bold{k}}}]^G_{\bold{k}}\to \bold{k}$ be a very generic character.
 Then the \'etale fundamental group of $\mathcal{N}_{\chi}$ has no $p'$ finite group quotients.

 \end{lemma}

\begin{proof}[Proof of Theorem \ref{g}.]

We start exactly as in the beginning of the proof of Theorem \ref{Main}. Since ${U}_{\chi}\mathfrak{g}$ is a simple $\mathbb{C}$-algebra, it follows from Lemma \ref{Susan} that $R$ is a Galois covering of
${U}_{\chi}\mathfrak{g}.$ Let $S$ be a  finitely generated subring of $\mathbb{C}$ over which $U_{\chi}\mathfrak{g}$ is defined, 
and let $R'$ be a $S$-subring of $R$ such that $R=R'\otimes_S\mathbb{C}$ and $R'$ is $\Gamma$-Galois covering
of ${U}_{\chi}(\mathfrak{g}).$ Let $S\to \bold{k}$ be a base change to algebraically closed field of characteristic $p>>0.$
Put $R_{\bold{k}}=R'\otimes\bold{k}.$ Since $U_{\chi}(\mathfrak{g}_{\bold{k}})$ is an Azumaya algebra over
$\mathcal{O}(\mathcal{N}_{\bar{\chi}})$  it follows from Proposition \ref{Azu} that  $\spec (Z(R_{\bold{k}}))\to \mathcal{N}_{\bar{\chi}}$ is a $\Gamma$-Galois
covering. 
But by Lemma \ref{Congugacy}, $\mathcal{N}_{\bar{\chi}}$ has no nontrivial $\Gamma$-Galois covering.
Therefore $$Z(R_{\bold{k}})\cong \Pi_{i=1}^{|\Gamma|} \mathcal{O}(\mathcal{N}_{\bar{\chi}})e_i,$$
where $e_i$ are orthogonal central idempotents $\sum_{i=1}^{|\Gamma|} e_i=1.$
Hence by Lemma \ref{Azu}, it follows that
$$R_{\bold{k}}\cong \Pi_{i=1}^{|\Gamma|}{U}_{\chi}(\mathfrak{g}_{\bold{k}})e_i.$$
In particular $R_{\bold{k}}$ has no nonzero nilpotent elements.

Next we claim that $R'$ is a Harish-Chandra bimodule over $\mathfrak{g}:$ the adjoint action of $\mathfrak{g}$ on $R'$
is locally finite.
Indeed, let $x\in R'$. It will suffice to show that  the nilradical of a Borel subalgebra $\mathfrak{n}\subset \mathfrak{g}$ acts
nilpotently on $x.$ Using the above isomorphism $R_{\bold{k}}\cong \Pi_{i=1}^nU_{\chi}(\mathfrak{g}_{\bold{k}})e_i$
we will introduce a PBW filtration on $R_{\bold{k}}$ which will restrict to the usual PBW
filtration on  $U_{\chi}(\mathfrak{g}_{\bold{k}}).$ Thus, 
$$\gr R_{\bold{k}}\cong \Pi_{i=1}^n\gr(U_{\chi}(\mathfrak{g}_{\bold{k}}))e_i,\quad  \deg(e_i)=0.$$
Just as in the proof of Theorem \ref{Main}, since $x$ is integral over $U_{\chi}\mathfrak{g},$ 
let $$x^m+\sum f_ix^{m-i}=0,\quad f_i\in U_{\chi}\mathfrak{g}.$$ Put $N=Max_i(\deg f_i)$, here $\deg f_i$ is understood as the filtration degree according
to the PBW filtration on $U_{\chi}(\mathfrak{g}_{\bold{k}}).$
Then it follows that $\deg (\bar{x})\leq N$ for all $p>>0,$ where $\bar{x}$ is the image of $x$ under the map 
$R\to R_{\chi}=\Pi_iU_{\chi}(\mathfrak{g}_{\bold{k}})e_i.$ Hence
$$x=\sum x_ie_i, \quad x_i\in U_{\chi}(\mathfrak{g})_{\bold{k}},\quad\deg(x_i)\leq N.$$
Let $l$ be such that $\ad(\mathfrak{n})^l(\mathfrak{g})=0$. Then $\ad(\mathfrak{n})^{lN}(x_i)=0,$ so $\ad(\mathfrak{n})^{lN}(\bar{x})=0$. Hence
 $\ad(\mathfrak{n})^{lN}(x)=0.$ Therefore the adjoint action of $\mathfrak{g}$ on $R'$ is locally finite as desired.
Now it follows that for any semi-simple $h\in \mathfrak{g}, R$ is $\mathbb{Z}$-graded according to
eigenvalues of $\ad(h)$ on $R.$ Let $h_1,\cdots, h_m\in \mathfrak{g}$ be semi-simple element 
such that they generate $\mathfrak{g}$ as a Lie algebra. 
By enlarging $S$ if necessary we may assume that $R'$ is $\mathbb{Z}^m$-graded algebra
according to  eigenvalues of $\ad(h_i), 1\leq i\leq m.$
Hence the degree 0 component  $R'$ is the centralizer $\mathfrak{g}$ in $R'.$
Thus $(R_{\bold{k}})$ is also $\mathbb{Z}^m$-graded.
  Since $e_i^2=e_i$ and  $R_{\bold{k}}$
has no nilpotent elements, this forces $e_i$ to have degree 0.
Denote the centralizer of $\mathfrak{g}$ in $R'$ (respectively in $R$) by $R'_0$ (respectively $R_0$).


Now we claim that $R_0$ is  commutative. Indeed, since $(R'_0)_{\bf{k}}$ is in the centralizer of $\mathfrak{g}$ in $R_{\bf{k}}$, which is 
$Z(U_{\chi}(\mathfrak{g}_{\bold{k}}))e_i,$ hence commutative. So $(R'_0)_{\bf{k}}$ is commutative for all $p>>0$, hence so is $R'_0$
and $R_0.$ We also have that $e_i\in  (R'_0)_{\bf{k}}, 1\leq i\leq n.$ In particular $\mathbb{C}\neq R_0.$
Thus $R_0$ is a commutative $\mathbb{C}$-domain, equipped with a $\mathbb{C}$-action of $\Gamma$, such that
$\mathbb{C}=R_0^{\Gamma}\neq R_0$. This is a contradiction.



\end{proof}

\noindent\textbf{Acknowledgement:} I am extremely grateful to A. Eshmatov for many useful discussion, particularly
for asking me the question whether the Weyl algebra can be Morita equivalent to a nontrivial skew group ring, answering which
led  to this paper. I would also like to thank the referee for many helpful suggestions.

\end{document}